\documentclass[english,a4paper,11pt]{article}
\usepackage{lmodern}

\usepackage{amssymb,amsmath,amsfonts,amsthm,epsfig,tikz,amscd} 
\usepackage{authblk}   
\usepackage{hyperref}
\usepackage{cite}

\usepackage[width=0.85\textwidth ]{caption}

\usepackage{hyperref}
\newtheorem{theorem}{Theorem}[section]

\newtheorem{prop}[theorem]{Proposition}

\title{Poncelet property of planar elliptic integrable Kepler billiards}

\author[1]{Daniel Jaud\footnote{Corresponding author}}
\affil[1]{\footnotesize Gymnasium Holzkirchen, Germany, Daniel.Jaud.PhD@gmail.com}

\author[2]{Lei Zhao}
\affil[2]{\footnotesize School of Mathematical Sciences, Dalian University of Technology, zhao1899@dlut.edu.cn}

\date{}

\begin{document}
\maketitle
\flushbottom


\begin{abstract}
\noindent
We consider the integrable dynamics of a Kepler billiard in the plane bounded by a branch of a conic section with one focus at the Kepler center. We show that in this case, for negative energy orbits associated with trajectories being ellipses, the lines of consecutive second orbital foci along a billiard trajectory are all tangent to a fixed circle. Based on this observation, we analyze in detail the integrable dynamics of a planar Kepler billiard, both inside and outside an elliptic reflection wall, with the Kepler center at one of its foci.  We identify the associated elliptic curve on which the dynamics are linearized, and the shift defined thereon. We also discuss explicit conditions on $n$-periodicity using Cayley's criteria. 

\end{abstract}

\vspace*{0.4cm}

\vspace*{0.4cm}
{\textbf{MSC-Classification:} 14H70, 37C79, 37J99, 37N05}


\section{Introduction}

The classical Poncelet theorem asserts that in the plane, given a pair of conic sections, a polygon inscribed in one conic section and circumscribing another always appears in an infinite family of such polygons \cite{Poncelet}. A modern review by Griffiths and Harris \cite{GH} explains that this system is ultimately linearized on an elliptic curve with a shift of finite order defined thereon. Starting from any point on the elliptic curve, we obtain, by successive applications of the shift, an orbit of finite order in the complex domain. When the entire orbit is real, the shift is real, and this corresponds to a polygon with the prescribed property. Shifting the initial point in the real part, we get an infinite family of polygons with the same property. The linkbetweene Poncelet porismando billiard dynamics is described i  \cite{Tabachnikov}.

Billiard systems in the plane are intensively studied in the research field of Dynamical Systems, both with and without a potential. Recently, many studies on the integrability and dynamics of billiards in a Kepler-Coulomb potential, often referred to as Kepler billiards, have been carried out. In \cite{Felder}, Felder integrated out the planar Kepler billiard with a line as the wall of reflection, whose integrability is due to Gallavotti and Jauslin \cite{GJ}, in terms of elliptic curves with a direct computation. Some conditions on periodic points have been computed in this work as well. Conditions on $n$-periodicity as well as the topology of such a system have been subsequently studied in \cite{GR}. Many more integrable Kepler billiards in the plane are pointed out in \cite{TZ1, TZ2}. 

This article aims to initiate an analysis on the dynamics of integrable Kepler billiard systems with a branch of a (non-degenerate) conic section as a reflection wall in the plane. In this system, the particle moves on one side of the reflection wall in a Kepler potential field whose center is at a focus of the conic section. As first observed in \cite{GJ} for the line case and later also in the cases with elliptic or hyperbolic reflection walls \cite{TZ1, TZ2, JZ}, along a billiard trajectory, the second orbital foci all lie on a circle centered at the second focus of the reflection wall. 

In the case of a line as a reflection wall, Felder \cite{Felder} has analyzed the integrable dynamics via the use of elliptic curves and elliptic functions. Analogously, it is natural to expect a link between the integrable dynamics of these billiard-type systems and elliptic curves. It is theoretically possible to identify this elliptic curve by directly parametrizing the elliptic boundary and the (elliptic) Keplerian orbits. We nevertheless adapt a different way based on a simple but for us unexpected geometric fact: Along a Kepler billiard orbit, the lines containing consecutive second orbital foci are all tangent to a circle, which we call the \emph{foci-caustic circle}. By considering how the second orbital foci is iterated, we reduce the analysis into the much more familiar, ``standard'' situation of the Poncelet porism. In this situation, the underlying elliptic curve, the shift, and the condition for $n$-periodicity are discussed in \cite{GH, GH2} in detail. An application of these classical works yields our desired results. We carry out the detailed analysis in this article for the Kepler problem with negative energy, where the point-like particle is reflected along an elliptic boundary. We leave the detailed analysis of several other cases for our future work.

In a recent preprint \cite{BBBT}, the authors showed in particular that at high energy, the only integrable Kepler billiards inside an ellipse are those with the Kepler center getting placed at a focus. Through the Hooke-Kepler correspondence, the billiard systems under our study are also directly related to the centered Hooke billiard inside an ellipse, for which some analysis is available \cite{Barrera, Panov, Fedorov}.

We organize this article as follows: In Section \ref{sec:Setup}, we set up our systems under study and discuss the foci-circle, a notion first introduced in \cite{J}. 
In Section \ref{sec: Foci} we analyse the geometric properties of the iterated second foci and show the existence of the foci-caustic circle. The billiard dynamics is now reduced to the dynamics of line segments tangent to a circle with end-points in another circle. The associated elliptic curve, the corresponding shift, and the relevant real component are computed in Section \ref{Sec: Elliptic Curve}. Conditions on $n$-periodicity are computed in \ref{Sec: Periodicity}, in which we also include some explicit discussions for small $n$.

\section{Integrable Kepler billiards and some of their geometrical properties}\label{sec:Setup}

We consider a billiard system in the plane $\mathbb{R}^{2}$, for which the particle moves under the influence of a Kepler(-Coulomb) potential field and gets reflected at a branch of a conic section $\mathbf{K}$ with a focus at the Kepler center $F$.  Let $F'$ be the second focus of the conic section boundary: This leaves out only the case of a parabola 
but includes the two important degenerate cases of a circle for which $F=F'$ and a line for which $F'$ is mirrored from $F$ through the line (the case implicitly investigated by Boltzmann \cite{B}). 

We set $F'=O$ as the origin of our coordinate system. We put  the Kepler center at $F=(-2 c_{\mathbf{K}}, 0)$, so $c_{\mathbf{K}} \ge 0$ is the linear eccentricity. The semi major axis of $\mathbf{K}$ is denoted by $a_{\mathbf{K}}$. The eccentricity of $\mathbf{K}$ is $e_{\mathbf{K}}=c_{\mathbf{K}}/a_{\mathbf{K}}$. The degenerate case of a circle and a line corresponds respectively to $c_{\mathbf{K}}=0$ and $a_{\mathbf{K}}=\infty$, two cases of no eccentricity.

The Keplerian orbits of the underlying system are conic section branches with a focus at $F$. At elastic reflections, the total energy is preserved, which gives us a conserved quantity of the system. We consider the cases where the total energy $E$ is non-zero. The orbital semi major axis is thus well-defined and is invariant under reflections. We denote it by $a$. Note that $a$ is related to the energy via $E=-\frac{\mu}{2a}$ with $\mu$ being the coupling constant of the gravitational field with $\mu<0$ for elliptical orbits and $\mu>0$ for hyperbolic orbits. The orbit is completely determined by its second focus up to orientation. Thus, to a billiard trajectory of Keplerian arcs reflecting at $\mathbf{K}$, there is an associated sequence of second foci $F_{1} F_{2} \dots $, which depends on the Kepler billiard trajectory and can be finite or infinite.

A major geometrical fact on the integrability of this mechanical billiard system is the following, which is established in \cite{GJ, Felder, TZ1, TZ2}:

\begin{prop} The sequence of second foci associated with a Kepler billiard trajectory lies on a circle centered at $F'$ in the elliptic/hyperbolic boundary cases.
\end{prop}

The circle is called the \emph{foci-circle} of the system, and is denoted by $S(F', R)$ for which $R >0$ is the radius \cite{JZ}. The quantity $R$ is another conserved quantity of the system. For a graphical representation of the setup, see Fig. \ref{fig_setup}.

A similar result holds for some specific limit of a parabolic boundary. We will not address this case in this article and will leave it for future work. 

\begin{figure}[htb]
\centering
\includegraphics[scale=0.9]{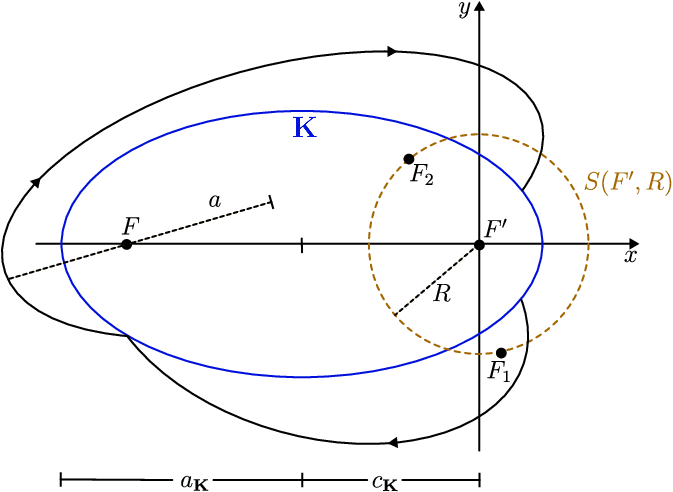}
\caption{Setup of the system with elliptic boundary and foci-circle on which all consecutive second Kepler foci lie.}
\label{fig_setup}
\end{figure}

\section{The foci-chords and the foci-caustic circle}\label{sec: Foci}

\subsection{The Foci-Caustic Circle}

Associated to a foci sequence $F_{1} F_{2} \dots$ there is a sequence of line segments $F_{1} F_{2}, F_{2} F_{3}, \dots$ that we call the \emph{foci-chords}. The lines containing them are denoted by $\ell(F_{i} F_{i+1}), i=1,2,3\dots$ respectively.
We have

\begin{theorem} Let the reflection wall be an ellipse or a branch of a hyperbola with one focus $F$ identical to the center of the Kepler force. Then along a Kepler billiard trajectory, the iterated line of foci-chords $\ell(F_{1} F_{2}), \ell(F_{2} F_{3}), \dots$ are tangent to the circle $S( x_{0}, |r_{0}|)$, called the \emph{foci-caustic circle}, in which $x_{0}=x_{0}(a, R)$ and $r_{0}=r_{0}(a, R)$ are given by quadratic functions in terms of the semi major axis $a$ of the Kepler orbits and the radius $R$ of the foci-circle.
\end{theorem}

For illustration, see Fig. \ref{fig:3_1}.
\begin{figure}[htb]
\centering
\includegraphics[scale=0.9]{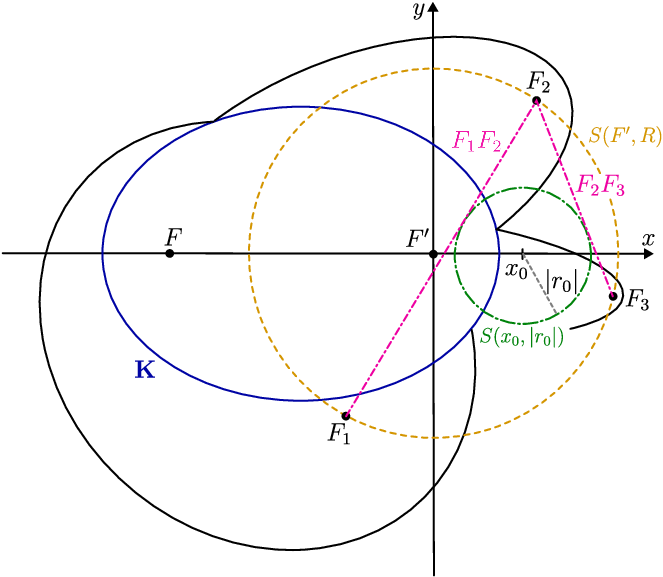}
\caption{Foci-chords tangent to the foci-caustic circle.}
\label{fig:3_1}
\end{figure}

Thus, the iterated foci-chords satisfy the usual setting of the Poncelet porism: The line containing the foci-chords are all tangent to the foci-caustic circle and have their ends in the foci-circle. 

To proceed, we distinguish cases: for the reflection boundary, we write $(\mathbf{E}), (\mathbf{H}_{F'}), (\mathbf{H}_{F})$ respectively for the cases of elliptic boundary, hyperbolic branch enclosing $F'$, and hyperbolic branch enclosing $F$. We also have respectively  $(\mathcal{E}), (\mathcal{H}_{F'}), (\mathcal{H}_{F})$ for orbit types: elliptic orbits, hyperbolic orbits enclosing $F'$, and hyperbolic orbits enclosing $F$. Note that hyperbolic orbits enclosing the empty focus $F'$ appear when the Kepler-Coulomb center is repulsive. 


\begin{proof} We consider the case that an orbit with second focus $F_{1}$ is reflected to an orbit with second focus $F_{2}$ at the point $P_{1} \in \mathbf{K}$, where the boundary $\mathbf{K}$ is an ellipse of a hyperbolic branch.

We set 
$$|FP_{1}|=A_{1}, \quad |F' P_{1}|=B_{1}, \quad |F_{1} P_{1}|=C_{1}^{+}, \quad |F_{2} P_{1}|=C_{1}^{-}.$$

Depending on the type of $\mathbf{K}$ we have by its two-foci definition
$$(\mathbf{E}): A_{1} + B_{1}=2 a_{\mathbf{K}};$$
$$(\mathbf{H}_{F'}): A_{1} - B_{1}=2 a_{\mathbf{K}};$$
$$ (\mathbf{H}_{F}): B_{1} - A_{1}=2 a_{\mathbf{K}}.$$

Then, for the orbits we have as well
$$(\mathcal{E}): A_{1} + C_{1}^{+}=A_{1} + C_{1}^{-}= 2 a;$$
$$(\mathcal{H}_{F'}): A_{1} - C_{1}^{-}=A_{1} - C_{2}= 2 a ;$$
$$(\mathcal{H}_{F}): A_{1} - C_{1}^{+}=A_{1} - C_{1}^{-}= -2 a.$$

In all cases
$$C_{1}^{+}=C_{1}^{-}:=C_{1},$$
and we obtain the following relationships in different cases:
$$(\mathbf{E}, \mathcal{E}): C_{1}-B_{1}= 2 (a-a_{\mathbf{K}});$$
$$(\mathbf{E}, \mathcal{H}_{F'}): C_{1} + B_{1}= 2 (a_{\mathbf{K}}-a);$$
$$(\mathbf{E}, \mathcal{F}): C_{1}+B_{1}= 2 (a+a_{\mathbf{K}});$$
$$(\mathbf{H}_{F'}, \mathcal{E}): C_{1}+B_{1}= 2 (a-a_{\mathbf{K}});$$
$$(\mathbf{H}_{F'}, \mathcal{H}_{F'}): C_{1}-B_{1}= 2 (a_{\mathbf{K}}-a);$$
$$(\mathbf{H}_{F'}, \mathcal{H}_{F}): C_{1}-B_{1}= 2 (a_{\mathbf{K}}+a);$$
$$(\mathbf{H}_{F}, \mathcal{E}): C_{1}+B_{1}= 2 (a+a_{\mathbf{K}});$$
$$(\mathbf{H}_{F}, \mathcal{H}_{F'}): B_{1}-C_{1}= 2 (a_{\mathbf{K}}+a);$$
$$(\mathbf{H}_{F}, \mathcal{H}_{F}): B_{1}-C_{1}= 2 (a_{\mathbf{K}}-a).$$

To continue, we denote the polar angle of $P_1$ by $\varphi$. The direction $F' P_1$ is thus given by the unit vector $(\cos \varphi, \sin \varphi)$. As mentioned in the proof of  \cite[Thm 2]{JZ}, $F_{1}, F_{2}$ are symmetric with respect to $F' P_1$. Thus the line $\ell(F_{1} F_{2})$ is perpendicular to $F' P_1$ and is given by an equation of the form
$$x\cos \varphi +y\sin \varphi +k=0,$$
which has signed distance $k$ to $F$.

 Write $h=|F_{1} F_{2}|/2$. By the construction of the foci-circle, we have
\begin{equation}\label{eq: BC}
{k^{2}+h^{2}=R^{2}.}
\end{equation}

Applying Pythagorean theorem to $P_{1}$, $F_{1}$ and the middle point 

\noindent
${S=(-k\cos(\varphi),-k\sin(\varphi))}$ of $F_{1} F_{2}$, {we have the formula}
\begin{equation}\label{eq: BCh}
{(B_{1} +k)^{2}+h^{2}=C_{1}^{2}.}
\end{equation}

{We thus get  from \eqref{eq: BC} and \eqref{eq: BCh} that
$$k=\dfrac{C_{1}^{2}-B_{1}^{2}-R^{2}}{2 B_{1}}.$$}





In any of the nine cases listed above, we have
$$C_{1}^{2}-B_{1}^{2}-R^{2}=(C_{1}-B_{1}) (C_{1}+B_{1})-R^{2}$$
$$=2 (\pm a_{\mathbf{K}} \pm a) \cdot (2 (\pm a_{\mathbf{K}} \pm a) \pm 2 B_{1})-R^{2}= 4 (a_{\mathbf{K}} \pm a)^{2}-R^{2} \pm 4 (a_{\mathbf{K}} \pm a) B_{1},$$
with a corresponding choice of signs, with the note that the signs are chosen in the same way in the expressions $(\pm a_{\mathbf{K}} \pm a)$.

We find
$$ k = \dfrac{4 (a_{\mathbf{K}} \pm a)^{2}-R^{2}}{2 B_{1}} \pm 2 (a_{\mathbf{K}} \pm a). $$

As $P_{1}$ lies on $\mathbf{K}$ we have 

$$B_{1}=\frac{a_{\mathbf{K}}^2\pm c_{\mathbf{K}}^2}{a_{\mathbf{K}} \pm c_{\mathbf{K}} \cdot \cos\varphi},$$

in which the sign in the numerator is $+$ when $\mathbf{K}$ is a hyperbolic branch, and is $-$ when $\mathbf{K}$ is an ellipse; the sign in the denominator is $-$ for $\mathbf{H}_{F}$ and is $+$ for the other two cases.

With this, we can now write
$$ k = \dfrac{(4 (a_{\mathbf{K}} \pm a)^{2}-R^{2})(a_{\mathbf{K}} \pm c_{\mathbf{K}} \cdot \cos \varphi)}{2 (a_{\mathbf{K}}^2\pm c_{\mathbf{K}}^2)} \pm 2 (a_{\mathbf{K}} \pm a) $$
and thus the equation of the line $\ell(F_{1} F_{2})$ can be written as
$$\Bigl(x {+} \dfrac{(4 (a_{\mathbf{K}} \pm a)^{2}-R^{2}) c_{\mathbf{K}}}{2 (a_{\mathbf{K}}^2\pm c_{\mathbf{K}}^2)} \Bigr) \cos \varphi +y \, \sin \varphi + \dfrac{a_{\mathbf{K}} (4 (a_{\mathbf{K}} \pm a)^{2}-R^{2})  }{2 (a_{\mathbf{K}}^2\pm c_{\mathbf{K}}^2)} \pm 2 (a_{\mathbf{K}} \pm a)=0,$$
in which the choice of the sign is the same in the expressions $(a_{\mathbf{K}} \pm a)$.

Now, if we write the above equation as 
$$(x - x_{0}) \cos \varphi + y \sin \varphi+r_{0}=0$$
then this line, and as well as any of the other lines $\ell(F_{i} F_{i+1}) \dots$, is tangent to the circle $S(x_{0}, r_{0})$ centered at $x_{0}$ with radius $|r_{0}|$, for which both $x_{0}$ and $r_{0}$ are uniquely determined by the reflection boundary $\mathbf{K}$ and the fixed conserved quantities $(a, R)$.  


\end{proof}



\subsection{The Foci and Foci-caustic Circles: The $(\mathbf{E}, \mathcal{E})$ Case}

We consider the $(\mathbf{E}, \mathcal{E})$ case in more detail. Recall that distinct ellipses with a focus in common intersect at at most two points. Thus, the particle may move either inside or outside $\mathbf{E}$ and leads to the same dynamics up to a time reversal \cite{JZ}.

\subsubsection{The Expressions of $x_{0}, r_{0}$ and the Parameter Range}

We have 
\begin{equation}\label{eq: x_{0} r_{0} EE }
x_{0}=-\dfrac{(4 (a-a_{\mathbf{K}})^{2} -R^{2}) c_{\mathbf{K}}}{2 (a_{\mathbf{K}}^2-c_{\mathbf{K}}^2)}, \, ~r_{0}=\dfrac{(4 a_{\mathbf{K}} (a-a_{\mathbf{K}})^{2} -a_{\mathbf{K}} R^{2}) }{2 (a_{\mathbf{K}}^2-c_{\mathbf{K}}^2)}+2 (a-a_{\mathbf{K}}).
\end{equation}


Once the elliptic boundary $\mathbf{K}$ of type $(\mathbf{E})$ is given, we have

\begin{prop}\label{prop: admissible EE} For the Kepler ellipses to perform physical reflections with the elliptical boundary, the conserved quantities $(a,  R) \in \mathbb{R}^{2}_{+}$ have to satisfy
$$2 a>a_{\mathbf{K}}-c_{\mathbf{K}}>0,\,\, |2 a_{\mathbf{K}}-2 a| \le R \le 2 a + 2 c_{\mathbf{K}}.$$
\end{prop}

For fixed $a_{\mathbf{K}}>c_{\mathbf{K}}>0$, we call such a pair of $(a, R)$ \emph{admissible}.

\begin{proof} The Hill's region is the projection in configuration space of the energy hypersurface. The particle is allowed to move only in this region. The Hill's region of the Kepler problem with semi major axis $a$ is the disk $D(F, 2 a)$ centered at $F$, which has to contain some point $\mathbf{K}$ and in particular its pericenter, which has distance $a_{\mathbf{K}}-c_{\mathbf{K}}$ from $F$. We thus have the first condition.

Also, the Hill's region $D(F, 2 a)$ has to contain the second orbital focus, so it has to intersect the foci circle $S(F', R)$. This is equivalent to saying that the closest point from $S(F', R)$ to $F$ is in $D(F, 2 a)$. We thus get 
$$-2 a \le 2 c_{\mathbf{K}}-R \le2 a$$
and thus 
$$2 (c_{\mathbf{K}}-a)<R<2 (a+c_{\mathbf{K}}).$$
On a point of reflection $P$, if a corresponding second orbital focus is $F_{1}$, then we have by the two-foci definition of ellipses 
$$|FP|+|F'P|=2 a_{\mathbf{K}},~~  |FP|+|F_{1}P|=2 a,$$
so we have 
$$R=|F'F_{1}| \ge 2 |a-a_{\mathbf{K}}|.$$  
As 
$$c_{\mathbf{K}}-a<a_{\mathbf{K}}-a \le  |a-a_{\mathbf{K}}|$$
we get the second condition. The combined conditions are represented in Fig. \ref{fig:region}. 
\end{proof}

This condition is also sufficient, in the sense that for any value of $(a, R)$ satisfying Prop. \ref{prop: admissible EE}, there exist billiard trajectories of elliptic Keplerian orbits with these parameter values. This follows from the following analysis on orbits of the second foci, which asserts that for all these parameter values there are orbits of the second foci $\{F_{i}\}_{i \in \mathbb{Z}}$ and the following simple observation: If $F_{i}$ and $F_{i+1}$ are both real, then by the previous construction, $F_{i}$ is reflected to $F_{i+1}$ by the line perpendicular to $F_{i} F_{i+1}$ and pass their middle point. Moreover, by construction, this line is tangent to $\mathbf{K}$ and thus, in particular, the corresponding elliptic Keplerian orbits are reflected to each other at the point of tangency. Thus, the analysis in the sequel faithfully reflects the dynamics of the Kepler billiard system.

\begin{figure}[htb]
\centering
\includegraphics[scale=0.63]{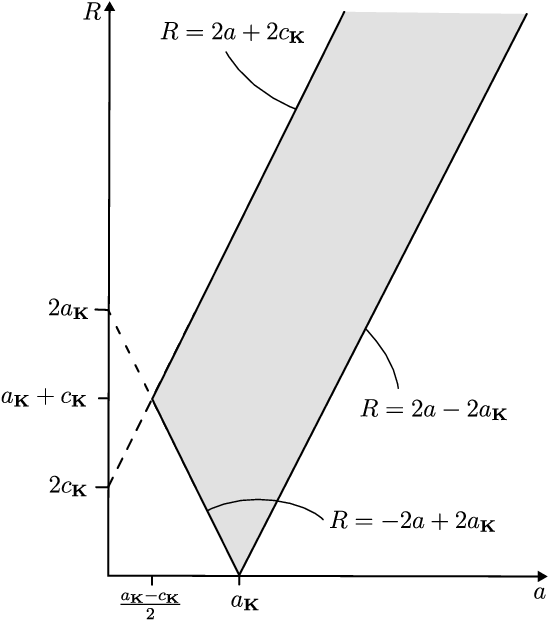}
\caption{Gray regions displaying the admissible values ($a,R$) for physical reflections with the elliptic boundary.}
\label{fig:region}
\end{figure}





\subsubsection{Scenarios of the Two Circles}\label{Subsection: Two Circles}

We fix an admissible $a > (a_{\mathbf{K}}-c_{\mathbf{K}})/2$ and increase $R$ from $0$. This results in different scenarios.

When $R=0$, we have
$$x_{0}=-\dfrac{2 (a-a_{\mathbf{K}})^{2} c_{\mathbf{K}}}{ (a_{\mathbf{K}}^2-c_{\mathbf{K}}^2)}, r_{0}=\dfrac{2 a_{\mathbf{K}} (a-a_{\mathbf{K}})^{2} }{ (a_{\mathbf{K}}^2-c_{\mathbf{K}}^2)}+2 (a-a_{\mathbf{K}}).$$
As $a_{\mathbf{K}}>c_{\mathbf{K}}$, we have $r_{0}>|x_{0}|$ and thus the degenerate $S(F', 0)=F'$ lies inside $S(x_{0}, r_{0})$, unless $a=a_{\mathbf{K}}$.

When $a \neq a_{\mathbf{K}}$, we increase $R$ from $0$: When $R$ is sufficiently small, the foci-caustic circle lies outside the foci-circle. This is the prohibited case, as no lines intersecting $S(F', R)$ can be tangent to  $S(x_{0}, r_{0})$. 

This situation does not change till $R=2 |a-a_{\mathbf{K}}|$, for which there hold $x_{0}=0, r_{0}= 2 (a-a_{K})>0$. In this case the two circles $S(F', R)$ and $S(x_{0}, r_{0})$ coincide. We have $R > 2 |a-a_{\mathbf{K}}|$ for admissible $R$ as in Prop. \ref{prop: admissible EE}. 

Moreover by \eqref{eq: x_{0} r_{0} EE }, for $R>2 |a-a_{\mathbf{K}}|$, always we have $x_{0}>0$.

For $R > 2 |a-a_{\mathbf{K}}|$, we consider the quantity 
$$x_{0}+|r_{0}|- R.$$

If $r_{0}>0$, then this quantity is
$$x_{0}+r_{0}-R=\dfrac{ (4 (a-a_{\mathbf{K}})^{2} - R^{2}) }{2 (a_{\mathbf{K}}+c_{\mathbf{K}})}+2 (a-a_{\mathbf{K}})-R$$
$$=\dfrac{ (2 (a-a_{\mathbf{K}})-R) (2 (a+c_{\mathbf{K}})-R) }{2 (a_{\mathbf{K}}+c_{\mathbf{K}})}.$$

If $r_{0}<0$, then it is
$$x_{0}-r_{0}-R=-\dfrac{ (4 (a-a_{\mathbf{K}})^{2} - R^{2}) }{2 (a_{\mathbf{K}}-c_{\mathbf{K}})}-2 (a-a_{\mathbf{K}})-R$$
$$=-\dfrac{ (2 (a-a_{\mathbf{K}})+R) (2 (a-c_{\mathbf{K}})-R) }{2 (a_{\mathbf{K}}-c_{\mathbf{K}})}.$$

We distinguish two cases:
\medskip

{\bf Case 1}: $a-c_{\mathbf{K}} \ge |a-a_{\mathbf{K}}|$, or equivalently $2 a \ge a_{\mathbf{K}}+c_{\mathbf{K}}$, since we have $a_{\mathbf{K}} > c_{\mathbf{K}} >0$.
\medskip

In this case, we have $x_{0}+|r_{0}|<R$ for $2 |a-a_{\mathbf{K}}|< R < 2 (a-c_{\mathbf{K}})$. As $x_{0}>0$, geometrically this implies that $S(x_{0}, r_{0})$ lies inside $S(F', R)$. 

In the particular case $r_{0}=0$, the foci-caustic circle degenerates into a point, and the corresponding foci-orbit are $2$-periodic. This is given by 
$$r_{0}:=\dfrac{(4 a_{\mathbf{K}} (a-a_{\mathbf{K}})^{2} -a_{\mathbf{K}} R^{2}) }{2 (a_{\mathbf{K}}^2-c_{\mathbf{K}}^2)}+2 (a-a_{\mathbf{K}})=0,$$

so
$$R=R_{2} := 2  \sqrt{(a-a_{\mathbf{K}}) ( a-c_{\mathbf{K}}^2/ a_{\mathbf{K}} )}. $$

As $c_{\mathbf{K}}< a_{\mathbf{K}}$, we have $R_{2}> 2 |a-a_{\mathbf{K}}|$. We also have $R_{2}< 2 |a-c_{\mathbf{K}}|$, which follows from the inequality

$$(a-a_{\mathbf{K}}) ( a-c_{\mathbf{K}}^2/ a_{\mathbf{K}} )<(a-c_{\mathbf{K}})^{2}.$$
This, after expansion, is equivalent to 
$$-  a_{\mathbf{K}} -  c_{\mathbf{K}}^2/ a_{\mathbf{K}} < -2  c_{\mathbf{K}},$$
which holds true by multiplying both sides by $a_{\mathbf{K}}$ and then completing the square.

We thus have $2 |a-a_{\mathbf{K}}|< R_{2} < 2 (a-c_{\mathbf{K}})$.

When $R=2 (a-c_{\mathbf{K}})$, we have

$$r_{0}=\dfrac{(4 a_{\mathbf{K}} (a-a_{\mathbf{K}})^{2} -4 a_{\mathbf{K}} (a-c_{\mathbf{K}})^{2}) }{2 (a_{\mathbf{K}}^2-c_{\mathbf{K}}^2)}+2 (a-a_{\mathbf{K}})=  \dfrac{ 2 a (c_{\mathbf{K}} -  a_{\mathbf{K}}) }{ a_{\mathbf{K}}+c_{\mathbf{K}}} <0,$$

$$x_{0}=\dfrac{2 c_{\mathbf{K}} ( 2 a-a_{\mathbf{K}} -c_{\mathbf{K}}) }{ a_{\mathbf{K}}+c_{\mathbf{K}}}.$$

For this case we have $x_{0}-r_{0}=x_{0}+|r_{0}|=R$, meaning that the two circles are tangent to each other, with   $S(x_{0}, r_{0})$ lies inside $S(F', R)$. 

Note that for this parameter value, there exists a particular focal reflection property \cite{JZ}: The ellipses tend monotonically to the degenerate limit of a collinear bouncing orbit at the pericenter of the boundary ellipse along its major axis. Now this property has an almost trivial geometrical proof: The iterated (oriented) foci-chords are well-defined outside the point of tangency of these two circles, and tend to the point of tangency.

For $2 (a-c_{\mathbf{K}}) < R < 2 (a+c_{\mathbf{K}})$, the two circles  $S(x_{0}, r_{0})$ and $S(F', R)$ intersect transversely. Only on the part of $S(F', R)$ lying outside $S(x_{0}, r_{0})$, there defines the dynamics of our iterated foci-chords. 


When $R=2 (a+c_{\mathbf{K}})$, we have $x_{0}+r_{0}=x_{0}-|r_{0}|=-R$. This means a tangency of the two circles $S(F', R)$ and $S(x_{0}, r_{0})$ at which $S(x_{0}, r_{0})$ encloses $S(F', R)$. 









The case that $R> 2(a+c_{\mathbf{K}})$, so that the two circles $S(F', R)$ and $S(x_{0}, r_{0})$ are separated, is disallowed by Prop. \ref{prop: admissible EE}.

\medskip

{\bf Case 2}: $a-c_{\mathbf{K}} < |a-a_{\mathbf{K}}|$, or equivalently $ a_{\mathbf{K}}-c_{\mathbf{K}} < 2 a < a_{\mathbf{K}}+c_{\mathbf{K}}$.

We have for $2 |a-a_{\mathbf{K}}|<R<2(a+c_{\mathbf{K}})$, there holds
$$x_{0}+r_{0}<R,~~~ x_{0}-r_{0}>R,$$
so the two circles  $S(x_{0}, r_{0})$ and $S(F', R)$ intersect transversely. 

For $R=2(a+c_{\mathbf{K}})$, the two circles are tangent to each other. The case  $R>2(a+c_{\mathbf{K}})$ is disallowed by Prop. \ref{prop: admissible EE}.

\section{The Associated Elliptic Curve of the $(\mathbf{E}, \mathcal{E})$ Case}\label{Sec: Elliptic Curve}

\subsection{The Griffiths-Harris Construction \cite{GH}}
Write $\mathcal{C}_{1}=S(O, R)$, $\mathcal{C}_{2}=S(x_{0}, r_{0})$ for the two circles, where $O=F'$. When considering periodic trajectories of an integrable Kepler billiard, being one of the cases listed before, considering foci-chords leads us to a particular case of the Poncelet porism. An intrinsic link between the Poncelet theorem and the theory of elliptic curves is pointed out by Griffiths-Harris in their classical work \cite{GH}. 

According to \cite{GH}, we should consider the lifted quadrics in $\mathbf{CP}^{2}$. We write $[z_{1}:z_{2}:z_{3}]$ for the homogeneous coordinates of points in  $\mathbf{CP}^{2}$. 
The complexifications of the curves $\mathcal{C}_{1}, \mathcal{C}_{2}$ are given respectively by the homogeneous equations
\begin{align}
    & {\cal{C}}_1: z_{1}^2+z_{2}^2-R^2 z_{3}^2=0,\\
    & {\cal{C}}_2: z_{1}^2-2 x_{0} \cdot z_{1} \cdot z_{3} +z_{2}^2+(x_{0}^2-r_{0}^{2}) z_{3}^2=0.
\end{align}

These are associated with the symmetric matrices
\begin{align}
   & Q_1[z_{1},z_{2},z_{3}]\mapsto Q_1=\begin{pmatrix}
        1 & 0 & 0\\
        0 & 1 & 0\\
        0 & 0 & -R^2
    \end{pmatrix},\\
  &  Q_2[z_{1},z_{2},z_{3}]\mapsto Q_2=\begin{pmatrix}
        1 & 0 & - x_{0}\\
        0 & 1 & 0\\
        - x_{0} & 0 & x_{0}^2-r_{0}^2
    \end{pmatrix}.
\end{align}

These curves are singular if and only if their associated symmetric matrices are singular, corresponding respectively to the cases $R=0$ and $r_{0}=0$. When $R=0$, the foci-circle degenerates into a point, and thus all points are fixed (1-periodic). This is possible only if $a=a_{\mathbf{K}}$. The case $r_{0}=0$ corresponds to the case when the foci-caustic circle degenerates into a point, and the corresponding orbits are all 2-periodic: This is the case $R=R_{2}$.

Otherwise, both curves are non-singular. From now on we assume $R \notin \{0, R_{2}\}$.

When $\mathcal{C}_{1}, \mathcal{C}_{2}$ are non-singular, their dual conics $\mathcal{C}^{*}_{1}$ and $\mathcal{C}^{*}_{2}$, consisting respectively of their tangent lines, are also non-singular in the dual projective space $\mathbf{CP}^{2 *}$, with associated symmetric matrices $Q_{1}^{-1}$ and $Q_{2}^{-1}$ respectively. 

Define the \emph{incidence variety} as
$$E= \{ (p, \xi) \in \mathcal{C}_2 \times \mathcal{C}_{1}^{*}, p \in \xi\}.$$
When $\mathcal{C}_{1}$ and $\mathcal{C}_{2}$ are transverse, $E$ is a smooth algebraic curve of genus 1 in $\mathbf{CP}^{2} \times \mathbf{CP}^{2 *}$. A proper choice of a base point as identity for the group law makes it into an elliptic curve.

A handy way to describe this elliptic curve is introduced in \cite{GH2}, which also leads to Cayley's criterion for $n$-periodic trajectories. This is to consider the following elliptic curve, which is birationally equivalent to $E$\footnote{Moreover, this birational equivalence preserves realness.},  defined through the pencil of quadrics through $\mathcal{C}_{1}$ and $\mathcal{C}_{2}$, with $\infty$ as the identity:
$$\mathcal{D}: y^{2}=-\det(t \cdot Q_1+Q_2).$$

Both $Q_{i}$ and $-Q_{i}$ define the same quadric. The negative sign is nevertheless preferred to help us track the real components of the associated elliptic curves.

We have
\begin{align}
t \cdot Q_1+Q_2:=  &\begin{pmatrix}
        t+1 & 0 & -x_{0}\\
        0 & t+1 & 0\\
        - x_{0} & 0 & -R^2 t+x_{0}^{2}-r_{0}^{2}
    \end{pmatrix},
\end{align}

whose determinant is computed as

\begin{equation}
\begin{split}
\det(t \cdot Q_1+Q_2) & = (t+1)^{2} (-R^{2} t+x_{0}^{2}-r_{0}^{2})-x_{0}^{2} (t+1)\\
 &=-R^{2}(t+1) \Bigl(   t^{2} -(\dfrac{x_{0}^{2}-r_{0}^{2}}{R^{2}}-1) t  + \dfrac{r_{0}^{2}}{R^{2}}\Bigr)\\
 & = -R^{2} (t+1)  (t+t_{1}) (t+t_{2}).
\end{split}
\end{equation}


We thus obtain the following expression for $\mathcal{D}$:
$$\mathcal{D}: y^{2}=R^{2} (t+1)  (t+t_{1}) (t+t_{2}),$$
in which
$$t_{1}=-\dfrac{1}{2}(\dfrac{x_{0}^{2}-r_{0}^{2}}{R^{2}}-1) - \dfrac{1}{2}\sqrt{(\dfrac{x_{0}^{2}-r_{0}^{2}}{R^{2}}-1)^{2}- 4 \dfrac{r_{0}^{2}}{R^{2}}},$$
$$t_{2}=-\dfrac{1}{2}(\dfrac{x_{0}^{2}-r_{0}^{2}}{R^{2}}-1) + \dfrac{1}{2}\sqrt{(\dfrac{x_{0}^{2}-r_{0}^{2}}{R^{2}}-1)^{2}- 4 \dfrac{r_{0}^{2}}{R^{2}}}.$$

\subsection{Discriminant and different cases of real forms}

The sign of the normalized discriminant of the quadratic equation
$$ (t+t_{1}) (t+t_{2})=0$$
is
$$\Delta:=(x_{0}^{2}-r_{0}^{2}-R^{2})^{2}- 4 r_{0}^{2} R^{2},$$
which determines the number of ramification points on the real part of $\mathcal{D}$. A short computation leads to the factorization
$$\Delta=(x_{0}+r_{0}-R) (x_{0}-r_{0}+R) (x_{0}-r_{0}-R)  (x_{0}+r_{0}+R). $$

For $2 |a-a_{\mathbf{K}}|< R < 2 (a-c_{\mathbf{K}})$, the circle $S(x_{0}, r_{0})$ lies inside $S(F', R)$. We thus have 
$$-R< x_{0} \pm r_{0}<R$$ and thus $\Delta>0$ in this case. 

In addition, this implies $x_{0}^{2}-r_{0}^{2}-R^{2}<0$ and thus $0<t_{1}<t_{2}$. As $t_{1} t_{2} =r_{0}^{2}/R^{2}<1$, necessarily $t_{1}<1$. A short computation shows in this case $t_{2}<1$.
This means that the curve is regular.  All three finite ramification points of $\mathcal{D}$ are real, and thus its real component is the union of two circles. 

For $R=2 (a-c_{\mathbf{K}})$, we have $t_{2}=1$, the circle $S(x_{0}, r_{0})$ is tangent from inside to $S(F', R)$. In this case, $\Delta=0$ and the elliptic curve $\mathcal{D}$ is singular: two circles are pinched at one point. 

For $2 (a-c_{\mathbf{K}})< R <2 (a+c_{\mathbf{K}})$, the two circles intersect transversally.  In this case we have $x_{0}>0, r_{0}<0$ and
$$x_{0}-r_{0}>R, x_{0}+r_{0}> -R, $$
with moreover
$$x_{0}-r_{0}>0>-R, x_{0}+r_{0} < R.$$

For this case, we thus have $\Delta<0$, and two of the three finite ramification points are not real. The real component is thus a circle.

We summarize these discussions in the following theorem:

\begin{theorem}  The incidence variety $E$ is birationally equivalent to the elliptic curve given by
$$\mathcal{D}': y^{2}=(t+1) (t+t_{1}) (t+t_{2})$$
in which
$$t_{1}=-\dfrac{1}{2}(\dfrac{x_{0}^{2}-r_{0}^{2}}{R^{2}}-1) - \dfrac{1}{2 R^{2}}\sqrt{\Delta},$$
$$t_{2}=-\dfrac{1}{2}(\dfrac{x_{0}^{2}-r_{0}^{2}}{R^{2}}-1) + \dfrac{1}{2 R^{2}}\sqrt{\Delta},$$
with 
$$\Delta:=(x_{0}^{2}-r_{0}^{2}-R^{2})^{2}- 4 r_{0}^{2} R^{2}=(x_{0}+r_{0}-R) (x_{0}-r_{0}+R) (x_{0}-r_{0}-R)  (x_{0}+r_{0}+R).$$

Moreover, the real part $E_{\mathbb{R}}$ of $E$ is mapped to the real component $\mathcal{D}_{\mathbb{R}}'$ of $\mathcal{D}'$, which we characterize as follows:
\begin{itemize}
\item $2 |a-a_{\mathbf{K}}|< R < 2 (a-c_{\mathbf{K}})$: the union of two circles;
\item $R=2 (a-c_{\mathbf{K}})$: two circles pinched at one point;
\item  $2 (a-c_{\mathbf{K}})< R <2 (a+c_{\mathbf{K}})$: a circle. 
\end{itemize}

\end{theorem} 

Next, we identify the corresponding shift, which linearizes the dynamics on the elliptic curve, which is denoted by $\sigma$. According to \cite{GH} this is given by one of the two points over $t=0$ in $\mathcal{D}$: We easily find that in $\mathcal{D}'$:
$$y^{2}=t_{1} t_{2}=r_{0}^{2}/R^{2}, \quad y=\pm r_{0}/R. $$

 When $2 |a-a_{\mathbf{K}}|< R < 2 (a-c_{\mathbf{K}})$, the real part $\mathcal{D}_{\mathbb{R}}'$ is defined over $(-1, -t_{2})$ and $(-t_{1}, \infty)$ and the shift $\sigma$ sends a point $t=\infty$ to a point over $t=0$ lying in the same circle. By the theory of elliptic curves, $\sigma$ is independent of the initial points, and thus always stays in the same circle component, so $\sigma$ is a rotation when restricted to one of the two circles. When $2 (a-c_{\mathbf{K}})< R <2 (a+c_{\mathbf{K}})$, the real part is a circle, with a nontrivial linear shift $\sigma$ defined there on.


\subsection{The Weierstrass Form}

We start with the curve
$$\mathcal{D}': y^{2}=(t+1) (t+t_{1}) (t+t_{2})$$
and mark the shift $\sigma$ represented by $\sigma (\infty, \infty)= (0, \pm  r_{0}/R)$ of our relevance.

Expanding the defining equation for $\mathcal{D}'$ we have



$$y^{2}=t^{3}+(t_{1}+t_{2}+1)t^{2} + (t_{1}+t_{2}+t_{1} t_{2})t + t_{1} t_{2}.$$

We set $t=\tilde{t}-\dfrac{t_{1}+t_{2}+1}{3}$ to eliminate the quadratic term. We get

$$y^{2}=\Bigl(\tilde{t}-\dfrac{t_{1}+t_{2}+1}{3}\Bigr)^{3}$$
$$+(t_{1}+t_{2}+1) \Bigl(\tilde{t}-\dfrac{t_{1}+t_{2}+1}{3})^{2} + (t_{1}+t_{2}+t_{1} t_{2}\Bigr) (\tilde{t}-\dfrac{t_{1}+t_{2}+1}{3}) + t_{1} t_{2}.$$

Expanding this equation and writing $y=\tilde{y}/2$, we obtain





$$\tilde{y}^{2}=4 \tilde{t}^{3} + 4 \Bigl((t_{1}+t_{2}+t_{1} t_{2}) - \dfrac{(t_{1}+t_{2}+1)^{2}}{3}\Bigr) \tilde{t}$$
$$ + 4 (\dfrac{2 (t_{1}+t_{2}+1)^{3}}{27}  -\dfrac{t_{1}+t_{2}+1}{3}(t_{1}+t_{2}+t_{1} t_{2})  + t_{1} t_{2}\Bigr),$$

which represents $\tilde{\mathcal{D}}$ in the Weierstrass form

$$\tilde{\mathcal{D}}: \tilde{y}^{2}=4 \tilde{t}^{3} -g_{2} \tilde{t}-g_{3} $$
with
$$g_{2}=-4 \Bigl((t_{1}+t_{2}+t_{1} t_{2}) - \dfrac{(t_{1}+t_{2}+1)^{2}}{3}\Bigr),$$
$$g_{3}=-4 \Bigl(\dfrac{2 (t_{1}+t_{2}+1)^{3}}{27}  -\dfrac{t_{1}+t_{2}+1}{3}(t_{1}+t_{2}+t_{1} t_{2})  + t_{1} t_{2}\Bigr).$$

This curve is parametrized by a Weierstrass $\wp$-function $\wp$ such that $(\tilde{t}, \tilde{y})=(\wp, \wp')$ is associated to a lattice $\Lambda \in \mathbb{C}$ uniquely determined by the equation of $\tilde{\mathcal{D}}$.

On the other hand from the equation of $\mathcal{D}'$ we have 



$$ \tilde{\mathcal{D}}: \tilde{y}^{2}= 4 (\tilde{t}-e_{1}) (\tilde{t}-e_{2}) (\tilde{t}-e_{3}),$$
in which 
$$e_{1}=\dfrac{t_{1}+t_{2}+1}{3}-1, e_{2}=\dfrac{t_{1}+t_{2}+1}{3}-t_{2}, e_{3}=\dfrac{t_{1}+t_{2}+1}{3}-t_{1}.$$
We have in any case $e_{1} \in \mathbb{R}$. When $\Delta>0$, we have 
$$ 0 < t_{1}<t_{2}<1,$$
 and thus $e_{1}<e_{2}<e_{3}$ in this case. When $\Delta<0$, $e_{2}$  is complex conjugate to $e_{3}$, while $e_{1}$ is real. 

A pair of fundamental periods $\{\omega_{1}, \omega_{2}\}$ of the lattice $\Lambda$ associated to the elliptic curve can be computed by integrating the invariant differential  $\dfrac{d \tilde{t}}{\sqrt{4 \tilde{t}^{3} -g_{2} \tilde{t}-g_{3} }}$  over the fundamental cycles. Explicitly, when $\Delta>0$, we may use the formulas
$$\omega_{1}/2=\int_{e_{3}}^{+\infty} \dfrac{d \tilde{t}}{\sqrt{4 \tilde{t}^{3} -g_{2} \tilde{t}-g_{3} }},$$
$$\omega_{2}/2=\int_{-\infty}^{e_{1}} \dfrac{d \tilde{t}}{\sqrt{4 \tilde{t}^{3} -g_{2} \tilde{t}-g_{3} }}$$
to compute the periods, in which $\omega_{1}$ is real and $\omega_{2}$ is purely imaginary. The fundamental domain is a rectangle. The real part of $\tilde{D}$ is the union of two circles, each having period $\omega_{1}$.


The shift on $\tilde{\mathcal{D}}$ is given by integration from the point $(\tilde{t}, \tilde{y})=(\infty, \infty)$ to $(\tilde{t}, \tilde{y})=(\frac{t_{1}+t_{2}+1}{3}, \pm 2 r_{0}/R)$ of the invariant differential $\dfrac{d \tilde{t}}{\sqrt{4 \tilde{t}^{3} -g_{2} \tilde{t}-g_{3} }}$. Explicitly, this shift can be computed as one of the two choices of $\pm \xi$, with:

$$\xi:= \int_{\frac{t_{1}+t_{2}+1}{3}}^{+\infty} \dfrac{d \tilde{t}}{\sqrt{4 \tilde{t}^{3} -g_{2} \tilde{t}-g_{3} }}.$$

As $\frac{t_{1}+t_{2}+1}{3}> e_{3}=\frac{t_{2}+1-2 t_{1}}{3}$, we have $\xi <\omega_{1}/2$, so shifting by $\xi$ does not exchange the two circles. 


When $\Delta<0$, we know that $e_{1}$ remains real,  $e_{2}, e_{3}$ are non-real and they are complex conjugate. In this case, the fundamental domain of the periodic lattice is a rhombus, generated by $\omega_{1} \in \mathbb{R}$ and $\omega_{2}$ such that $\Re \{\omega_{2}/\omega_{1}\}=1/2$. The real part of  $\tilde{\mathcal{D}}$ is a circle, whose period is computed from 
$$\omega_{1}/2=\int_{e_{1}}^{+\infty} \dfrac{d \tilde{t}}{\sqrt{4 \tilde{t}^{3} -g_{2} \tilde{t}-g_{3} }}. $$
A complex half-period can be computed by 
$$\int_{e_{2}}^{e_{3}} \dfrac{d \tilde{t}}{\sqrt{4 \tilde{t}^{3} -g_{2} \tilde{t}-g_{3} }} $$
along a (vertical) path in $\mathbb{C}$ connecting $e_{2}$ to $e_{3}$.

The associated shift is $\pm \xi$ with 
$$\xi:= \int_{\frac{t_{1}+t_{2}+1}{3}}^{+\infty} \dfrac{d \tilde{t}}{\sqrt{4 \tilde{t}^{3} -g_{2} \tilde{t}-g_{3} }}.$$

As now $e_{1}<\frac{t_{1}+t_{2}+1}{3}$, the shift is within a half real period. 

\textit{Remark:}  It is not hard to see that the frequency, namely the ratio shift/period is a non-constant function of the parameters, indeed this follows from the observation that when $R=2 (a-c_{mathbf{K}})$ there exists 1-periodic orbits which cease to exists in some other cases. From this it is in principle possible to draw KAM type conclusion for sufficiently small perturbations of our system, just as remarked by Felder \cite{Felder} in his study of the line boundary case. As we are dealing with integrable systems, we shall not treat detailed applications of KAM theory here.

\section{Condition on $n$-Periodicity}\label{Sec: Periodicity}
\subsection{ Existence of $n$-periodic orbits in the elliptic case}
We follow the classical argument presented in \cite{Tabachnikov} for the existence of $n$-periodic orbits in the $(\mathbf{E}, \mathcal{E})$ case. The following proposition is rather weak and can certainly be generalized to many other cases, while here we merely want to establish the basic fact of their existence.  

As in \cite{Tabachnikov}, we consider the cyclic configurations space of $n$ points on $\mathbf{E}$:
$$G(\mathbf{E}, n)=\{(q_{1}, q_{2}, \dots, q_{n}): q_{i} \in \mathbf{E}, \, q_{i} \neq q_{j} \,\, \forall i \neq j\}.$$
This space is not connected. Its connected components are further labelled by their rotation number, which is an integer counting how many times the closed collection of chords $\{q_{1} q_{2}, q_{2} q_{3}, \dots, q_{n} q_{1}\}$ winds around $\mathbf{E}$.

\begin{prop}\label{prop: periodic orbits existence} In the $(\mathbf{E}, \mathcal{E})$ case, for a given orbital semi major axis $a$ such that $a > a_{\mathbf{K}}+c_{\mathbf{K}}$, for any integer $n \ge 2$, there exists $R$ such that all the orbits of $(a, R)$ are $n$-periodic.
\end{prop}

\begin{proof}
We transform the problem into the equivalent mechanical billiard problem of an attractive Hooke center lying at the center of an elliptic reflection wall, still denoted by $\mathbf{E}$, based on the complex square mapping 
$$\mathbb{C} \to \mathbb{C}, z \mapsto z^{2}.$$ 
See \cite{TZ1, JZ} for details. In particular, the negative of the energy of the Kepler problem appears as the string constant in the corresponding Hooke problem, having normalized energy 1. 
The relevant corresponding cyclic configuration space is the symmetric cyclic configuration space 
$$\tilde{G}(\mathbf{E}, n)=\{z_{1}, -z_{1}, z_{2}, -z_{2}, \dots , z_{n}, -z_{n}, z_{i}^{2} \neq z_{j}^{2} \,\,\, \forall \,\, i \neq j\},$$
in which $z_{i}^{2}=q_{i}, i=1,2,\cdots n.$

The energy being given, by assumption $\mathbf{E}$, lies entirely in the corresponding Hill's region of the planar Hooke problem. Thus, the corresponding Jacobi metric is a non-degenerate Riemannian metric in the region bounded by $\mathbf{E}$, and the billiard system is equivalent up to a time change to the billiard system defined by geodesics of the Jacobi metric.  We denote by $d_{J} (x, y)$ the Jacobi-distance between two points $x, y$. For the latter, we consider the maximum of the total length function $\sum_{i=1}^{n} d_{J} (z_{i}, z_{i+1})$ with the natural convention $z_{n+1}:=z_{1}$. By triangle inequality, the maximum must be achieved in the interior of a connected component of $\tilde{G}(\mathbf{E}, n)$ and thus gives an $n$-periodic orbit, which in turn gives an $n$-periodic orbit of the Kepler billiard we have started with. Finally, due to its link with the Poncelet porism, for the value of $(a, R)$ corresponds to such an $n$-periodic orbit, all orbits are $n$-periodic.

\end{proof}

In the above proposition, the assertion $a > a_{\mathbf{K}}+c_{\mathbf{K}}$ is rather strong. It is only made to have $\mathbf{E}$ lying entirely in the interior of the Hill's region, for which we do not have to analyse the intersection of the zero-velocity curve (the boundary of the Hill's region) with $\mathbf{E}$, and thus the argument takes its simplest form. On the other hand, already the analysis of the $2$-periodic case in Sec \ref{Subsection: Two Circles}, which only exists in Case 1, shows that when $\mathbf{E}$ intersects the zero-velocity curve, the situation is more involved.

\subsection{Cayley's Criteria for $n$-Periodicity}
Explicit conditions on Poncelet $n$-gon have been found by Cayley and reviewed in modern terms by Griffiths-Harris \cite{GH}. This can be applied directly in our setting. 

To state the criteria, we expand the function $\sqrt{\det(t\cdot C_1+C_2)}$ into a Taylor series around $t=0$:
$$\sqrt{\det(t\cdot C_1+C_2)}=A_0+A_1t+A_2t^2+\dots $$

Then the condition for having a Poncelet $n-$gon for $n\geq 3$, circumscribed by ${\cal{C}}_2$ and inscribed in ${\cal{C}}_1$, is
\begin{align}
    \det \begin{pmatrix}
        A_2 & A_{3} \dots & A_{m+1}\\
        A_3 & A_{4} \dots & A_{m+2}\\
        \dots & \dots & \dots \\
        A_{m+1} & \dots & A_{2m}
    \end{pmatrix}&=0~~~\mbox{for odd}~n=2m+1,\\
    \det \begin{pmatrix}
        A_3 & A_{4} & \dots & A_{m+1}\\
        A_4 & A_{5} & \dots & A_{m+2}\\
        \dots & \dots & \dots \\
        A_{m+1} & A_{m+2} &  \dots & A_{2m-1}
    \end{pmatrix}&=0~~~\mbox{for even}~n=2m.\\
\end{align}


\subsubsection{Computations of $A_{i}$'s}

From previous discussions, we have
\begin{align}
t \cdot Q_1+Q_2:=  &\begin{pmatrix}
        t+1 & 0 & -x_{0}\\
        0 & t+1 & 0\\
        - x_{0} & 0 & -R^2 t+x_{0}^{2}-r_{0}^{2}
    \end{pmatrix},
\end{align}
whose determinant is
$$\det(t \cdot Q_1+Q_2)=-R^{2} (t+1) \Bigl(   t^{2} -(\dfrac{x_{0}^{2}-r_{0}^{2}}{R^{2}}-1) t  + \dfrac{r_{0}^{2}}{R^{2}}\Bigr).$$

We do not need to bother with the negative factor $-1$ in the discussions of periodicity. We set
$$F= R^{2}(t+1) \Bigl(   t^{2} -(\dfrac{x_{0}^{2}-r_{0}^{2}}{R^{2}}-1) t  + \dfrac{r_{0}^{2}}{R^{2}}\Bigr).$$



We have 

$$F=\sum_{i=0}^{3} C_{i} t^{i},$$
with

$$C_{0}=r_{0}^{2}, \, C_{1}=( 2 r_{0}^{2}+R^{2}-x_{0}^{2}), \, C_{2}=2 R^{2}-x_{0}^{2}+r_{0}^{2}, \, C_{3}=R^{2}.$$

If we set
$$\sqrt{F}=\sum_{i=0}^{\infty} A_{i} t^{i},$$
then by comparing coefficients, we obtain

\[
\begin{aligned}
& A_{0}^{2}=C_{0}, \\ 
& A_{0} A_{1}=C_{1}, \\
& 2 A_{0} A_{2} + A_{1}^{2}=C_{2}, \\
& 2 A_{0} A_{3}+2 A_{1} A_{2}=C_{3}, \\
& 2 A_{0} A_{4}+2 A_{1} A_{3}+A_{2}^{2}=0, \\
& 2 A_{0} A_{5}+2 A_{1} A_{4}+2 A_{2} A_{3}=0, \\
& 2 A_{0} A_{6}+2 A_{1} A_{5}+2 A_{2} A_{4}+A_{3}^{2}=0, \\
& \cdots \cdots \cdots
\end{aligned}
\]


After choosing one of the two roots from the first relation, the other coefficients are determined consecutively in an explicit manner. 

The two choices of $A_{0}$ result in just a sign change in all that follows, so we may just choose $A_{0}=r_{0}$. We thus find:

\[
\begin{aligned}
& A_{1}=C_{1}/ (2 A_{0})=\dfrac{2 r_{0}^{2}+R^{2}-x_{0}^{2}}{2 r_{0}},\\ 
& A_{2}=\dfrac{C_{2}-A_{1}^{2}}{2 A_{0}}
\end{aligned}
\]



$$A_{3}=\dfrac{C_{3}-2 A_{1} A_{2}}{2 A_{0}}=\dfrac{R^{2}-2 A_{1} A_{2}}{2 r_{0}}.$$\

Substituting $A_{1}$ and $A_{2}$ in we get
$$A_{3}=\dfrac{8 R^{2} r_{0}^{4}- (2 r_{0}^{2}+R^{2}-x_{0}^{2}) (4 R^{2} r_{0}^{2}  - (x_{0}^{2}-R^{2})^{2})}{16 r_{0}^{5}}.$$

We may then find the other $A_{i}'s$ iteratively:

\[
\begin{split}
& A_{4}=-\dfrac{2 A_{1} A_{3}+A_{2}^{2}}{2 A_{0}}, \\ 
& A_{5}=-\dfrac{2 A_{1} A_{4}+2 A_{2} A_{3}}{2 A_{0}}, \\
& A_{6}=-\dfrac{2 A_{1} A_{5}+2 A_{2} A_{4}+A_{3}^{2}}{2 A_{0}}. \\
\end{split}
\]




In general for $m \ge 2$, we have

$$A_{2m}=-\dfrac{2 A_{1} A_{2m-1}+2 A_{2} A_{2m-2}+\cdots +A_{m}^{2}}{2 A_{0}},$$
and

$$A_{2m+1}=-\dfrac{2 A_{1} A_{2m}+2 A_{2} A_{2m-1}+\cdots +A_{m} A_{m+1}}{2 A_{0}}.$$


Below, we analyse the case corresponding to small $n$:

\subsubsection{$n=2$}

The case $n=2$ corresponds to the degenerate case when the foci-caustic circle becomes a point, and has been considered before in Sec \ref{Subsection: Two Circles}. Recall that the condition $R=R_{2}$ is obtained from
$$r_{0}:=\dfrac{(4 a_{\mathbf{K}} (a-a_{\mathbf{K}})^{2} -a_{\mathbf{K}} R^{2}) }{2 (a_{\mathbf{K}}^2-c_{\mathbf{K}}^2)}+2 (a-a_{\mathbf{K}})=0.$$
This results into a real $R_{2}= 2  \sqrt{(a-a_{\mathbf{K}}) ( a-c_{\mathbf{K}}^2/ a_{\mathbf{K}} )}$ only in Case 1 in Sec \ref{Subsection: Two Circles}. Fig. \ref{fig:period2} displays two Poncelet configurations for a 2-periodic scenario.

\begin{figure}[htb]
\centering
\includegraphics[scale=0.75]{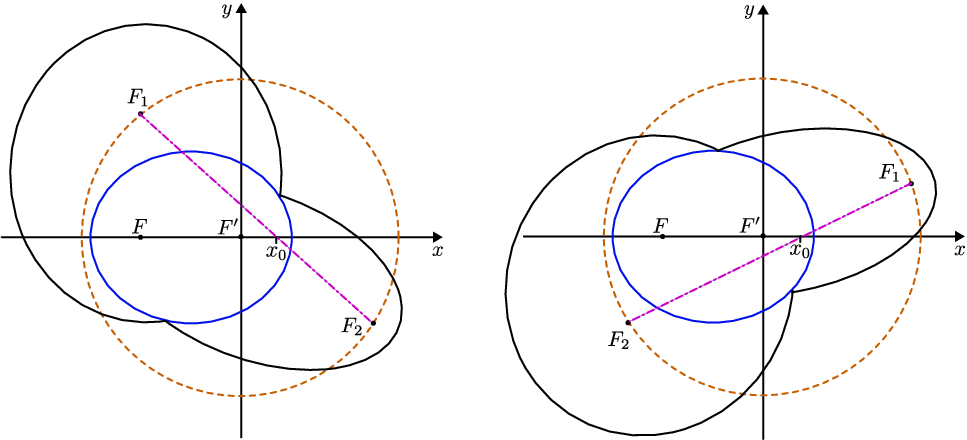}
\caption{2-periodic orbits and the Poncelet property.}
\label{fig:period2}
\end{figure}

\subsubsection{$n=3$}

The condition on $3$-periodic points is
$$A_{2}
=\dfrac{4 R^{2} r_{0}^{2} - (x_{0}^{2}-R^{2})^{2}}{8 r_{0}^{3}}=0,$$
which is equivalent to
\begin{equation}\label{eq: factor 3-periodic}
(2 R r_{0} + x_{0}^{2}-R^{2})(2 R r_{0} - x_{0}^{2}+R^{2})=0,
\end{equation}
so in general, we have two sets of conditions for 3-periodicity, obtained by equalizing each factor to 0.

From \eqref{eq: x_{0} r_{0} EE } we have
$$x_{0}=-\dfrac{(4 (a-a_{\mathbf{K}})^{2} -R^{2}) c_{\mathbf{K}}}{2 (a_{\mathbf{K}}^2-c_{\mathbf{K}}^2)},$$

and

$$ r_{0}=\dfrac{(4 a_{\mathbf{K}} (a-a_{\mathbf{K}})^{2} -a_{\mathbf{K}} R^{2}) }{2 (a_{\mathbf{K}}^2-c_{\mathbf{K}}^2)}+2 (a-a_{\mathbf{K}})=\dfrac{4 (a-a_{\mathbf{K}}) (a_{\mathbf{K}} a - c_{\mathbf{K}}^{2})-a_{\mathbf{K}} R^{2}}{2 (a_{\mathbf{K}}^2-c_{\mathbf{K}}^2)}.$$
Substituting these in each factor of \eqref{eq: factor 3-periodic}, we obtain the condition that one of the following two equations should be satisfied:
$$\dfrac{(4(a-a_{\mathbf{K}})^{2}-R^{2})^{2} c_{\mathbf{K}}^{2}}{4 (a_{\mathbf{K}}^2-c_{\mathbf{K}}^2)^{2}}-R^{2} \pm R \cdot \dfrac{4(a-a_{\mathbf{K}}) (a_{\mathbf{K}} a-c_{\mathbf{K}}^{2})-a_{\mathbf{K}} R^{2}}{a_{\mathbf{K}}^2-c_{\mathbf{K}}^2}=0.$$

We know from Prop \ref{prop: periodic orbits existence} that $3$-periodic orbits exist for $a > a_{\mathbf{K}}+c_{\mathbf{K}}$. This should not be optimal for the parameter range for the existence of $3$-periodic orbits, but it does not seem either easy or illuminating to identify the precise threshold. In Fig. \ref{fig:period3}, one can see a 3-periodic orbit including the foci-circle, the foci-caustic circle, and the tangent lines formed by the consecutive second Kepler foci.

\begin{figure}[htb]
\centering
\includegraphics[scale=0.8]{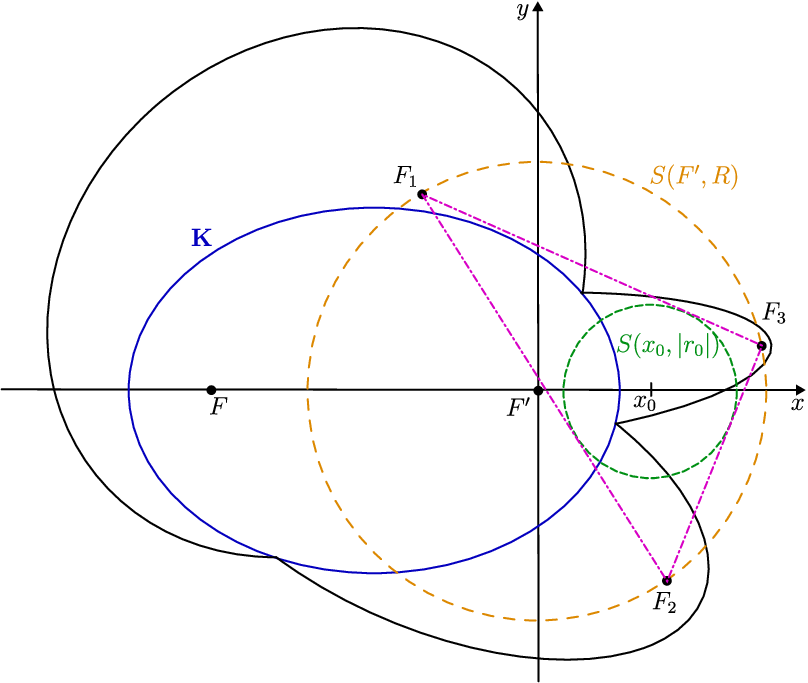}
\caption{Example for a 3-periodic configuration.}
\label{fig:period3}
\end{figure}

\subsubsection{$n=4$}

For $4$-periodic orbits the condition is $A_{3}=0$, which is equivalent to the following condition :

$$(R^{2}-x_{0}^{2}) (2 r_{0}^{2} (R^{2}+x_{0}^{2})-(R^{2}-x_{0}^{2})^{2})=0.$$

We obtain $4$-periodic orbits when one of the following two conditions holds
$$R=|x_{0}|,$$
$$2 r_{0}^{2} (R^{2}+x_{0}^{2})=(R^{2}-x_{0}^{2})^{2}.$$
The second equation is cubic in $R^{2},$ and it does not seem easy to make a detailed analysis.

The equation 
$$R=|x_{0}|,$$
can be solved explicitly as 
$$R=\frac{a_\mathbf{K}^2-c_\mathbf{K}^2}{c_\mathbf{K}}\cdot \left(1+\sqrt{1+\frac{4c_\mathbf{K}^2(a-a_\mathbf{K})^2}{(a_\mathbf{K}^2-c_\mathbf{K}^2)^2}}\right).$$

As the parameter range is given by
$$2 a>a_{\mathbf{K}}-c_{\mathbf{K}}>0,\,\, |2 a_{\mathbf{K}}-2 a| \le R \le 2 a + 2 c_{\mathbf{K}},$$
We have
$$|2 a_{\mathbf{K}}-2 a| \le \frac{a_\mathbf{K}^2-c_\mathbf{K}^2}{c_\mathbf{K}}\cdot \left(1+\sqrt{1+\frac{4 c_\mathbf{K}^2(a-a_\mathbf{K})^2}{(a_\mathbf{K}^2-c_\mathbf{K}^2)^2}}\right) \le 2 a + 2 c_{\mathbf{K}}.$$
For the inequality on the left, after squaring both sides, we easily see that it always holds. For the inequality on the right, after squaring both sides and simplifying the expressions, we obtain the much simplified equivalent form, giving a restriction for the shape of the reflection wall:
$$3\, c_\mathbf{K}\ge a_\mathbf{K}.$$

\section*{Acknowledgement} 
We thank the anonymous reviewers for their work and useful suggestions to improve the presentation of this article.

L.Z. is supported by DFG ZH 605-4/1 and Research Funds for Central Universities of China.

\bibliographystyle{cas-model2-names} 

\end{document}